\documentclass[a4paper,12pt]{article}

\usepackage{amsmath}
\usepackage{amsthm}
\usepackage{amsfonts}
\usepackage{amssymb}
\usepackage{subeqnarray}
\usepackage{mathrsfs}
\usepackage{latexsym}

\usepackage[margin=2cm]{geometry}

\usepackage{graphicx}
\usepackage{color}
\usepackage{pstricks,psfrag}
\usepackage{epsfig}
\usepackage{subfigure}

\newcommand{\ud}{\mathrm{d}}
\newcommand{\beq}{\begin{equation}}
\newcommand{\eeq}{\end{equation}}
\newcommand{\ba}{\begin{eqnarray}}
\newcommand{\ea}{\end{eqnarray}}
\newcommand{\bs}{\boldsymbol}

\newtheorem{theorem}{Theorem}[section]
\newtheorem{remark}{Remark}[section]

\newtheorem{proposition}{Proposition}[section]
\newtheorem{lemma}{Lemma}[section]
\newtheorem{definition}{Definition}[section]

\title{The steady state configurational distribution diffusion equation of the standard FENE dumbbell polymer model: existence and uniqueness of solutions for arbitrary velocity gradients.}

\author{Ionel Sorin Ciuperca $^1$ and Liviu Iulian Palade $^2$}

\begin{document}

\maketitle

\begin{flushleft}

Universit\'e de Lyon, CNRS\\

$^1$ Universit\'e Lyon 1, Institut Camille Jordan UMR5208, B\^at Braconnier, 43 Boulevard du 11 Novembre 1918, F-69622, Villeurbanne, France.

$^2$ INSA-Lyon, Institut Camille Jordan UMR5208 \& P\^ole de Math\'ematiques, B\^at. Leonard de Vinci No. 401, 21 Avenue Jean Capelle, F-69621, Villeurbanne, France.

\end{flushleft} 

\begin{abstract}

The configurational distribution function, solution of an evolution (diffusion) equation of the Fokker-Planck-Smoluchowski type, is (at least part of) the corner stone of polymer dynamics: it is the key to calculating the stress tensor components.  This can be reckoned from \cite{bird2}, where a wealth of calculation details is presented regarding various polymer chain models and their ability to accurately predict viscoelastic flows.  One of the simplest polymer chain idealization is the Bird and Warner's model of finitely extensible nonlinear elastic (FENE) chains.  In this work we offer a proof  that the steady state configurational distribution equation has unique solutions irrespective of the (outer) flow velocity gradients (i.e. for both slow and fast flows).  

\end{abstract} 

\begin{flushleft}

Keywords:  FENE dumbbell chains; Fokker-Planck-Smoluchowski equation; existence and uniqueness of solutions; slow and fast viscoelastic flows;  Krein-Rutman theorems; \\

\end{flushleft}

\section{Introduction}\label{intr}

The viscoelastic flow behavior of polymeric liquids is strongly influenced by the complexity of various inter and intra molecular interactions. At microscopic level, long chain entanglements are a consequence of chain connectivity and backbone uncrossability due to intermolecular repulsive exclusive volume forces. Macromolecules diffusion (and conformational relaxation) is slowed down due to hydrodynamic drag and Brownian forces.   

Bird, Curtiss, Armstrong and Hassager, together with their collaborators (see \cite{bird2} and references cited therein), enriched significantly Kirkwood's early ideas \cite{kirk1} and produced a general kinetical theoretical framework for both diluted and concentrated polymeric systems.  Here, the macromolecules are modeled as freely jointed bead-rod or bead-spring chains.  One of the simplest version of this chain model is the (now popular) Bird - Warner's elastic dumbbell chain, that consists of two beads connected by a {\bf F}initely {\bf E}xtensible {\bf N}onlinear {\bf E}lastic - aka FENE - spring.  The salient features of this model, of relevance to this work, are briefly reviewed below, for sake of clarity.  

Let $\tilde{{\bf x}}\in\mathbb{R}^n,\,n=2,3$, denote the (microscopic)
dumbbell connector vector, ${\bf y}\in\mathbb{R}^n$ the (macroscopic)
Eulerian position vector.  In the absence of inertia and of external
forces, the balance of hydrodynamic, Brownian and intermolecular
forces results in the so-called Fokker-Planck-Smoluchowski 
with Dirichlet boundary condition for the
 configurational 
function 
\\
$\tilde{\psi}(t,{\bf y},\tilde{{\bf x}})$ diffusion equation.  In dimensionless form it looks:
\begin{subeqnarray}
& & \frac{\partial \tilde{\psi}}{\partial t}+{\bf u}\cdot\nabla_{\bf y}\tilde{\psi}=\nabla_{\tilde{{\bf x}}}\cdot\left[-{\bs \theta}\tilde{{\bf x}}\tilde{\psi} +\frac{1}{2\text{De}}\nabla_{\tilde{{\bf x}}}\tilde{\psi}+\frac{1}{2\text{De}}{\bf F}^{(c)}(\tilde{{\bf x}})\tilde{\psi}\right],\,\text{over}\,B(0,\tilde{\delta}) \slabel{d1}\\
& & \tilde{\psi}|_{\partial B(0,\tilde{\delta})}=0 \slabel{d2}
\end{subeqnarray}
In the above equation, $B(0,\tilde{\delta})$ is the open ball of
radius $\tilde{\delta}$ centered at $0$,  De is the Deborah number
and   ${\bs \theta}=(\nabla_{\bf y} {\bf u})^{T}$ is a tensor which represents
the (macroscopic) velocity gradient; the corresponding term accounts for the flow type.  The second term in the $rhs$ represents the statistically averaged Brownian force due to thermal fluctuations in the liquid.  The last term, ${\bf F}^{(c)}$, is the elastic force that accounts for the dumbbell's elastic response to  strain input, for which Warner \cite{war1} proposed the following expression (valid for $\|{\bf x}\|<\tilde{\delta}$, with $\tilde{\delta}$ a polymer depending parameter):
\beq\label{war}
{\bf F}^{(c)}(\tilde{{\bf x}})=\frac{\tilde{{\bf x}}}{1-(\|\tilde{{\bf x}}\|/\tilde{\delta})^{2}}
\eeq 
The above is commonly called the FENE force.  Now, as an aside, the model is quite flexible in that it may sustain other types of elastic forces: e.g. Peterlin's force (actually a linearized version of eq.\eqref{war}) usually referred to as FENE-P (see \cite{chil,her}):
\beq\label{petl}
{\bf F}^{(c)}(\tilde{{\bf x}})=\frac{\tilde{{\bf x}}}{1-<\|\tilde{{\bf x}}\|^{2}>/\tilde{\delta}^{2}}=\frac{\tilde{{\bf x}}}{1-<\textrm{tr}(\tilde{{\bf x}}\otimes\tilde{{\bf x}})>/\tilde{\delta}^{2}}
\eeq 
Asymptotic solutions to the diffusion equation are known for some steady state flows: see \cite{bird2} (for concise presentations see  \cite{lar, ber, hui1, faith, ott1}).  They were obtained through series expansions about the (known) equilibrium function $\tilde{\psi}_{eq}(\tilde{\bf{x}})$. 

Next, let $(t,{\bf y})\in \mathbb{R}_+\times\left(Q\subset\mathbb{R}^n \right)$.  The momentum balance equation reads (see \cite{bird1}):
\begin{subeqnarray}  
& & \frac{\partial {\bf u}}{\partial t}+\left( {\bf u}\cdot\nabla_{\bf
    y}\right) {\bf u}= \nu \Delta u -\nabla_{\bf y}p+\nabla_{\bf
  y}\cdot{\bf S},
\quad \text{over}\, \mathbb{R}_+\times Q \slabel{mb1}\\
& & \nabla_{\bf y}\cdot {\bf u}=0,\quad \text{over}\, \mathbb{R}_+\times Q \slabel{mb2}
\end{subeqnarray}
where $\nu > 0, \;
{\bf u}={\bf u}(t,{\bf y})\in\mathbb{R}^n$, $p=p(t,{\bf y})\in \mathbb{R}$.  ${\bf S}(t,{\bf y})\in \text{Sym}(\mathbb{R})$ is the symmetric extra stress tensor given by (\cite{bird2}):

\beq\label{ce}
{\bf S}(t,{\bf y})= \mu \left [ \int_{B(0,\tilde{\delta})}\tilde{{\bf
    x}}\otimes{\bf F}(\tilde{{\bf x}})\tilde{\psi}(t,{\bf
  y},\tilde{{\bf x}})\ud \tilde{{\bf x}}-
\int_{B(0,\tilde{\delta})} \tilde{\psi}(t,{\bf
  y},\tilde{{\bf x}})\ud \tilde{{\bf x}} \; {\bf I} \right ]
\eeq
where $\mu > 0 $ is a fluid related parameter (actually a given constant).

One observes that whenever the velocity gradient is such that
$\partial u_i/\partial y_j=a_{ij}$ = constant, $\sum_i a_{ii}=0 $, and 
$\tilde{\psi}$ is a solution of
\eqref{d1}-\eqref{d2}, then $ {\bf S}$  defined in equation \eqref{ce} is always
independent of ${\bf y}$, hence $\nabla_{{\bf
      y}}\cdot{\bf S}=0$.  
In such a situation there exist ${\bf u}$ and $p$ 
so that \eqref{mb1}-\eqref{mb2} are solved.  That this is
indeed the case may be inferred from the following.  Using Einstein's
summation convention over dummy indices, $u_i=a_{ij}y_j+c$, therefore
$\displaystyle \frac{\partial u_i}{\partial y_k}
u_k=a_{ik}\left[a_{kj}y_j+c \right]=a_{ik}a_{kj}y_j+d_i$,
$d_i=c\sum_{i,k} a_{ik}$.  Hence $\displaystyle\left[
  \nabla_{{\bf y}}\cdot({\bf u}\otimes{\bf
    u})\right]_{i}=a_{ik} a_{kj} y_j+d_i=\alpha_{ij}y_j+d_i$.
Therefore $\nabla_{{\bf y}}\cdot({\bf u}\otimes{\bf u})$ may
be expressed as $\nabla_{{\bf y}}\cdot({\bf u}\otimes{\bf
  u})=-\nabla_{{\bf y}}p$, where $p=-(1/2)\alpha_{ij}y_i
y_j-d_i y_i$, since the matrix of entries $\alpha_{ij}$ is symmetric.
We conclude that for any traceless matrix 
${\bf A}$ whose entries $a_{ij}$ are constants, and for a steady
state, homogeneous flow solution $\tilde{\psi}(\tilde{{\bf x}})$ - i.e. independent of
$t$ and 
${\bf y} $ - to 
equations \eqref{d1}-\eqref{d2},
there exists a steady state solution to \eqref{mb1}-\eqref{mb2}
 given by:
\begin{subeqnarray}
u_i(\tilde{{\bf x}}) & = & a_{ij}y_j+c  \slabel{sss1} \\
p & = & -\frac{1}{2}a_{ik}a_{kj}y_i y_j-c\sum_{i,k} a_{ik} y_i \slabel{sss2}
\end{subeqnarray}
and with ${\bf S}$ given by eq\eqref{ce}.  

For this work we shall consider ${\bf u}$ as being given by
eq\eqref{sss1}, where ${\bf A}$ is a given matrix, and we shall prove
the existence of a solution $\tilde{\psi}$, independent of $t$ and
${\bf y}$, to \eqref{d1}-\eqref{d2}.  

Before proceeding further, we pause for the following important
observation.  The solution $\tilde{\psi}$ to 
\eqref{d1}-\eqref{d2} we inquire
about - being a probability density - has to be non-trivial
($\tilde{\psi}\neq0$), non-negative and integrable.  As
$\tilde{\psi}=0$ is a solution to the aforementioned problem and as we
have to mind about non-trivial ones, the solution non-uniqueness must
be compulsory.  Next, we know from \cite{gil-tru, caj} that
$\tilde{\psi}=0$ is the unique solution to 
\eqref{d1}-\eqref{d2} whenever ${\bf F}^{(c)}$ is an element of 
$L^r(B(0,\tilde{\delta})),\,r>n$.  Therefore, what makes possible 
the existence of non-trivial solutions, is the fact 
that ${\bf F}^{(c)}$ is NOT an element of
$L^r(B(0,\tilde{\delta})),\,r>n$
(in fact ${\bf F}^{(c)}$ is not an element of
$L^r(B(0,\tilde{\delta}))$ for any $ r \geq 1$).

Now, as ${\bf F}^{(c)}=\nabla_{\tilde{{\bf x}}}U(\tilde{{\bf x}})$,
with $U(\tilde{{\bf x}})=-\left(\tilde{\delta}^{2}/2\right)\log
\left(1-\|\tilde{{\bf x}}\|^2/\tilde{\delta}^{2}\right) $,
equation  \eqref{d1} is usually re-written 
as (see \cite{dlp}):
\beq\label{dd}
 -\frac{1}{2\text{De}}\nabla_{\tilde{{\bf
       x}}}\cdot\left[\tilde{M}(\tilde{{\bf x}})\nabla_{\tilde{{\bf
         x}}}\left(\frac{\tilde{\psi}}{\tilde{M}} 
\right)\right]+\nabla_{\tilde{{\bf x}}}\cdot\left[{\bs\theta }\tilde{{\bf x}}\tilde{\psi}\right]=0
\eeq
where the function $\tilde{M}:B(0,\tilde{\delta})\to\mathbb{R}$ is given by:
\beq
\tilde{M}\left(\tilde{{\bf x}} \right)=\frac{1}{J}\left(1-\frac{\|\tilde{{\bf x}}\|^2}{\tilde{\delta}^2} \right)^{\tilde{\delta}^2/2}
\eeq
where $J $ is a normalization constant so that:
\beq\label{bvp2}
\int_{B(0,\tilde{\delta})}\tilde{M}\left(\tilde{{\bf x}} \right) \ud \tilde{{\bf x}} =1.
\eeq
Next, for sake of generality, we replace ${\bs \theta}\tilde{{\bf x}}$ by an arbitrary function $\tilde{{\bf k}}:B(0,\tilde{\delta})\to\mathbb{R}^n$, and the problem we focus on can be re-formulated as:
\begin{subeqnarray}
& & -\frac{1}{2\text{De}} \nabla\cdot \left[ \tilde{M}\nabla\left( \frac{\tilde{\psi}}{\tilde{M}}\right) \right]+\nabla\cdot\left(\tilde{{\bf k}}\tilde{\psi} \right)=0, \quad \forall\tilde{{\bf x}} \in  B(0,\tilde{\delta}) \slabel{bp31}\\
& & \tilde{\psi}|_{\partial B(0,\tilde{\delta})}=0 \slabel{bp32}\\
& & \tilde{\psi}\geq 0 \slabel{bp33}\\
& &
\int_{B(0,\tilde{\delta})}\tilde{\psi}\left(\tilde{{\bf
      x}} \right) \ud \tilde{{\bf x}} = a \slabel{bp34}
\end{subeqnarray}
where $a=\dfrac{1}{\text{meas}(Q)}$ is a given constant. 

Next, for notation convenience,
we carry out the variable change ${\bf x}=\tilde{{\bf
    x}}/\tilde{\delta}$.  This transforms the domain $\tilde{\Omega}$
into $\Omega=B(0,1)=\left\{ {\bf x}\in\mathbb{R}^n,\, \|{\bf x}\|\leq
  1 \right\}$.  
Let us denote $\psi({\bf x})=\tilde{\psi}(\tilde{{\bf x}}) $, 
${\bf k}({\bf x})=  2 \tilde{\delta} \, \text{De} \, \tilde{{\bf k}}(\tilde{{\bf x}}) $, 
$ \delta = \tilde{\delta}^2/2 $ and let
\begin{equation*}
M : \Omega \rightarrow \mathbb{R} , \;
M({\bf x}) = (1 - \|{\bf x}\|^2)^\delta. 
\end{equation*}
Then, equations \eqref{bp31}-\eqref{bp34} become in $\Omega$:
\begin{subeqnarray}
& & - \nabla\cdot \left[ M\nabla\left( \frac{\psi}{M}\right) \right]+\nabla\cdot\left({\bf k}\psi \right)=0, \quad \forall{\bf x} \in  \Omega 
\slabel{dvp31}\\
& & \psi|_{\partial \Omega} =0 \slabel{dvp32}\\
& & \psi\geq 0 \slabel{dvp33}\\
& &
\int_\Omega \psi ({\bf
      x})  \ud {\bf x} = b \slabel{dvp34}
\end{subeqnarray}
with $ b > 0 $ and $ {\bf k} : \Omega \rightarrow \mathbb{R}^n $ given.
As in practical situations $\tilde{\delta}$ is (roughly speaking) close to 10,
then $\delta$ is close to 50.  

The goal of this paper is to prove the existence and uniqueness
of a solution to the system of equations \eqref{dvp31}-\eqref{dvp34}.  
We easily
see that the aforementioned problem can be also formulated as following:
prove that $0$ is a simple eigenvalue of the operator (denoted from
now on $L$) defined by the lhs of \eqref{dvp31} and the boundary
condition \eqref{dvp32}, with a corresponding non-negative and
integrable eigenvector.  In fact, we will prove that $0$ is the
principal eigenvalue of $L$ in the sense that the real part of any
other eigenvalue of it is non-negative.  To achieve this we use the
classical Krein-Rutman theorems, in both weak and strong senses, on an
appropriate operator obtained from $L$.  
This will also entail that $\psi$ is positive over $\Omega$ and behaves like $M$ on the boundary $\partial\Omega$.  

The boundary value problem problem with unknowns ${\bf u}$ and $\tilde{\psi}$ as
presented in \eqref{d1}, \eqref{d2}, \eqref{mb1}, \eqref{mb2}, \eqref{ce}
has attracted the attention of several investigators working in the area.
For instance, in \cite{zh} Zhang and Zhang proved the 
existence of a local in time, regular solution to the system formed by equations
\eqref{d1}, \eqref{d2}, \eqref{mb1}, \eqref{mb2}, and \eqref{ce}.  The existence of a global in time solution  was proved in \cite{lzz-08} by
Lin et al, and that in a particular case referred to as the
``co-rotational''
velocity field, that is, in equation \eqref{d1} the term $ {\bs \theta} = (\nabla u)^T $
is replaced by $ {\bs \theta} = \nabla_y u - (\nabla_y u)^T $.  Moreover, for this same system of equations, in \cite{ba}, 
Barret et al offered a proof for the existence and uniqueness of a solution to a regularized problem associated to
the system \eqref{d1}, \eqref{d2}, \eqref{mb1},\eqref{mb2}, \eqref{ce}.
Next, in \cite{jlb-04} Leli\`evre et al proved the existence and uniqueness of a local in time solution to the one dimensional motion system of equations in which the  Fokker-Planck-Smoluchowski equation is
replaced by a stochastic diffusion differential equation.

In \cite{du} Du et al focused on the
Fokker-Plank-Smoluchowski evolution equation only, assuming a steady and homogeneous 
macroscopic velocity field. For this they proved the
global in time existence and uniqueness of a solution. For the corresponding steady state problem, the forementioned authors 
proved the existence of a solution only in the particular case where 
the tensor ${\bs \theta}$ in \eqref{d1} is either symmetric or antisymmetric.  Some of their numerical simulations suggest  the existence of steady-state solutions for arbitrary  ${\bs \theta}$.  

In this work we do prove the existence and uniqueness of steady state solutions for arbitrary ${\bs \theta}$.  

As an aside, in \cite{dlp} Degond et
al provided arguments in support of the validity of an asymptotic expansion 
solution, valid for small De numbers, first obtained in \cite{bird2}.  

This paper is organized as follows: 

\begin{itemize}
\item in Section  \ref{main} we state the main steady state  existence and uniqueness result,
\item Section \ref{prel} addresses some important functional analysis preliminaries,  \item Section \ref{proof} is devoted to proving the conclusive existence and uniqueness result.
\end{itemize}

\section{Functional framework. Presentation of the main result.}
\label{main}

Let the following  spaces be defined as:
\begin{eqnarray}\label{ff1}
& & L_M^2 \equiv L_M^2(\Omega):=\left \{ u\in L^1_{\text{loc}}(\Omega),\, \int_\Omega \frac{u^2}{M} \ud {\bf x}<\infty \right \} \\
& & H_M^1 \equiv H_M^1(\Omega):=\left \{ u\in L^1_{\text{loc}}(\Omega),\, \int_\Omega 
\left [\frac{u^2}{M}+M\left|\nabla \left(\frac{u}{M} \right)\right|^2 \right ] \ud {\bf x}<\infty \right \}
\end{eqnarray}
endowed with the norms
$$
\left ( \int_\Omega \frac {u^2} M \ud  {\bf x} \right )^{1/2} 
\quad \text{ and respectively }
\quad \left (
\int_\Omega 
\left [\frac{u^2}{M}+M\left|\nabla \left(\frac{u}{M} \right)\right|^2
\right ] \ud {\bf x} \right )^{1/2}.
  $$

It is clear that $L_M^2$ is a Hilbert space.
To see that $H_M^1$ is also a Hilbert space, let us remark that 
$$
H_M^1 = M \cdot {\hat H}_M^1 \quad \text { with } \quad
{\hat H}_M^1 = \left \{ v\in L^1_{\text{loc}}(\Omega),\, \int_\Omega 
(M {v^2}+M\left|\nabla v \right|^2 ) \ud {\bf x}<\infty \right \}.
 $$

It is well-known, as being a classical weighted Sobolev space, that ${\hat H}_M^1 $ is complete (see for exemple Theorem 3.2.2.(a) in Triebel's monograph \cite{trie}) when endowed with the norm

$$
\left ( \int_\Omega 
(M {v^2}+M\left|\nabla v \right|^2 ) \ud {\bf x} \right )^{1/2}.
  $$
Since the application $ \psi \in {\hat H}_M^1  \rightarrow M \psi \in
H_M^1 $ is an isometry, we deduce that $H_M^1$ is complete. 

For any $\varphi\in H_M^1(\Omega) $ we denote $|\varphi|_1$ the semi-norm on $H_M^1$ defined by
$|\varphi|_1^2:=\int_{\Omega}M\left|\nabla\displaystyle \frac{\varphi}{M} \right|^2 \ud {\bf x}$.  
Moreover, 
$\left( H_M^1(\Omega) \right)'$ denotes the corresponding dual space and one has the canonical embedding
$$
L_M^2 \subset \left( H_M^1(\Omega) \right)'.
  $$
We now endeavor to search for solutions to equations
\eqref{dvp31}-\eqref{dvp34} that are 
elements of $H_M^1$, as the trace on $\partial\Omega$
of any $ u\in H_M^1(\Omega)$ is zero (see also Proposition \ref{sfs27}).  

To achieve this goal, equation \eqref{dvp31} is first multiplied 
by $\varphi/M$, with $\varphi\in\mathcal{D}(\Omega)$ 
and next integrated over $\Omega$.  It gives:  

\beq\label{ff2}
\int_{\Omega}M \nabla\left( \frac{\psi}{M} \right) \cdot \nabla\left( \frac{\varphi}{M} \right)\ud {\bf x}
-\int_{\Omega}{\bf k}\psi \cdot \nabla\left( \frac{\varphi}{M} \right) \ud {\bf x} =0
\eeq

\begin{definition}\label{ff3}
$\psi\in H_M^1(\Omega)$ is a weak solution of the system 
\eqref{dvp31}-\eqref{dvp34}, provided that:
\beq\label{ff4}
\int_{\Omega} M \nabla\left( \frac{\psi}{M} \right) \cdot 
\nabla\left( \frac{\varphi}{M} \right)\ud {\bf x}
-\int_{\Omega}{\bf k}\psi \cdot \nabla\left( \frac{\varphi}{M} 
\right) \ud {\bf x} =0,\quad \forall\varphi\in H_M^1(\Omega)
\eeq
and moreover, \eqref{dvp33} and \eqref{dvp34}  are satisfied.

\end{definition}
Next, let the operator $L:H_M^1(\Omega)\rightarrow\left(H_M^1(\Omega)\right)' $ be defined as:
\beq\label{ff5}
\left\langle L(u),\varphi \right\rangle := \int_{\Omega} M 
\nabla\left( \frac{u}{M} \right) \cdot \nabla\left( 
\frac{\varphi}{M} \right)\ud {\bf x}
-\int_{\Omega}{\bf k} u \cdot \nabla\left( \frac{\varphi}{M} 
\right) \ud {\bf x} ,\quad \forall\varphi\in H_M^1(\Omega)
\eeq
Now, $L$ is well-defined, due to:
\begin{equation}
\nonumber
\begin{split}
& \left| \int_{\Omega} M \nabla\left( \frac{u}{M} \right) \cdot \nabla\left( \frac{\varphi}{M} \right)\ud {\bf x}
-\int_{\Omega}{\bf k} u \cdot \nabla\left( \frac{\varphi}{M} \right) \ud {\bf x} \right| \leq \\
& \leq \int_{\Omega} \left| M^{1/2} \nabla\left( \frac{u}{M} \right) \right| \cdot \left| M^{1/2} \nabla\left( \frac{\varphi}{M} \right) \right|\ud {\bf x}
+ \int_{\Omega} \left | {\bf k} \right | \left| \frac{u}{M^{1/2}} \right| M^{1/2} \left| \nabla\left( \frac{\varphi}{M} \right)\right| \ud {\bf x} \\
& \leq \|u\|_{H_M^1}\|\varphi\|_{H_M^1}+\left\| {\bf k}\right\|_{L^{\infty}}\|u\|_{L_M^2}\|\varphi\|_{H_M^1}
\end{split}
\end{equation}
It is now clear that our problem is tantamount to finding 
an element $\psi \in H^1_M$ such that
\begin{subeqnarray}
& & L \psi =0 \slabel{vp1}\\
& & \psi\geq 0 \slabel{vp2}\\
& &
\int_\Omega \psi ({\bf
      x})  \ud {\bf x} = b,\slabel{vp3}
\end{subeqnarray}
that is,  $\psi$ must by a non-negative and integrable eigenvector of 
$L$ corresponding to the eigenvalue 0.
\\
For any $\beta\geq0$, let:
\beq\label{sfs2}
X_\beta:=\left\lbrace  \varphi\in\mathcal{C}\left(
    \overline{\Omega}\right) ,
\,\exists c\geq0 \quad\text{s.t.}\quad|\varphi({\bf x})|\leq 
cM^\beta({\bf x}),\,\forall{\bf x}\in\Omega  \right\rbrace.
\eeq

$X_\beta$ is a Banach space endowed with the norm

\beq\label{sfs3}
\|\varphi\|_{X_\beta}:=\displaystyle\mathop{\sup}_{{\bf x}\in\Omega}
\frac{|\varphi({\bf x})|}{M^\beta({\bf x}) }=\inf\left\lbrace c\geq0 
\quad\text{s.t.}\quad|\varphi({\bf x})|\leq cM^\beta({\bf x}),
\,\forall{\bf x}\in\Omega \right\rbrace. 
\eeq

\begin{remark}\label{sfs4}
$X_0= \mathcal{C}^0\left( \overline{\Omega}\right)$, endowed with the usual norm $\|\varphi\|_{X_0}=\displaystyle\mathop{\sup}_{{\bf x}\in\Omega}|\varphi|$.   
\end{remark}

\begin{remark}\label{sfs5}
For $\beta_1<\beta_2$, the continuous inclusion
$X_{\beta_2}\displaystyle\mathop{\subset}_{\text{cont}} 
X_{\beta_1}$ holds true. 
\end{remark}
Let the cone $P_\beta\subset X_\beta$ be defined as:

\beq\label{sfs8}
P_\beta:=\left\lbrace \varphi\in X_\beta,\, \varphi({\bf x})\geq0,\, \forall {\bf x}\in\Omega  \right\rbrace.
\eeq 

It is clear that $P_\beta$ is a reproducible cone for the space
$X_\beta$, that is \ $X_\beta = P_\beta - P_\beta $.
\\
It can be easily seen the interior
$\stackrel{\circ}{P_\beta}$ of $P_\beta$ is given by:  
\begin{eqnarray}\label{sfs10}
\stackrel{\circ}{P_\beta} & = & \left\lbrace \varphi\in
  X_\beta\quad\text{s.t.}\quad\displaystyle\mathop{\inf}_{{\bf
      x}\in\Omega}\frac{\varphi({\bf x})}{M^\beta({\bf x})}>0
\right\rbrace
= \left\{ \varphi\in X_\beta,\, \exists c_1>0 \quad\text{s.t.}
\quad \varphi({\bf x})\geq c_1M^\beta({\bf x}),\,\forall{\bf
  x}\in\Omega   
\right\} \nonumber\\
& = & \big\{ \varphi\in \mathscr{C}(\Omega) ,\, \exists c_1,c_2; \;
0<c_1<c_2, \quad\text{s.t.}\quad c_1M^\beta({\bf
  x})\leq\varphi({\bf x})\leq c_2M^\beta({\bf x}), \;
\forall{\bf x}\in\Omega   \big\}.
\end{eqnarray}

We now state the cornerstone result of this paper:
\begin{theorem}\label{est9-1}[\textbf{Existence and uniqueness theorem}]
Let $ b > 0, \; \delta \geq 8$ and $ {\bf k} \in \left ( W^{1, \infty}
(\Omega) \right )^n.$ Then there exists an unique solution $\psi$ to
the system (\eqref{vp1}-\eqref{vp3}).
Moreover, this solution belongs to $\stackrel{\circ}{P_1}$ which
amounts to say that $\psi$ is continuous in $\Omega$, and
there exist $ c_1, \; c_2 $ with \ $ 0 < c_1 < c_2 $ \ such that
$$
c_1 M({\bf x}) \leq \psi({\bf x}) \leq  c_2 M({\bf x}), 
\quad \forall \; {\bf x} \in \Omega.
  $$
\end{theorem}

\begin{remark}
We assume throughout this paper that $\delta$ and ${\bf k}$
comply with the hypotheses of Theorem \ref{est9-1}.  Given the physical model under consideration (for which we gave a suitable description in the Introduction section), such an assumption does not lower down the level of generality.

\end{remark}

\section{Several preliminary results}
\label{prel}

\subsection{Basic facts}

In the following we denote  for any  real 
$\alpha$ the operator $L_\alpha : H_M^1(\Omega)\rightarrow \left( H_M^1(\Omega) \right)'$
given by $L_\alpha=L+\alpha I_d$, $I_d$ being the identity operator.

It is assumed, throughout this paper, that 
$\alpha$ is large enough so that:
\beq
\label{hyp-alp}
\alpha \geq \max \left \{ \frac 1 2 \|{\bf k}\|_{L^\infty(\Omega)} + 1, \quad
4 \lambda_0^2 + \lambda_0 n + 2 \lambda_0 
\|{\bf k}\|_{L^\infty(\Omega)}
+ \|{\nabla \cdot \bf k}\|_{L^\infty(\Omega)} \right \}
\eeq
where  
\beq
\label{lamb0}
\lambda_0 = 2 (\|{\bf k}\|_{L^\infty(\Omega)} + 1)
\eeq
\begin{proposition}\label{ff8}
The operator $L_\alpha$ is invertible.  
\end{proposition}

\begin{proof}
Let $f\in \left( H_M^1(\Omega) \right)'$, arbitrary.  We have to prove the existence of a unique solution $u\in H^1_M$ to
\beq\label{ff9}
a_\alpha(u,\varphi)=\langle f,\varphi \rangle, \, \forall \varphi \in H_M^1
\eeq
where, in the above,
\beq\label{ff10}
a_\alpha(u,\varphi)=\int_{\Omega} M \nabla\left( \frac{u}{M} \right) \cdot \nabla\left( \frac{\varphi}{M} \right)\ud {\bf x}
-\int_{\Omega}{\bf k} u \cdot \nabla\left( \frac{\varphi}{M} \right) \ud {\bf x} +\alpha\int_\Omega \frac{u\varphi}{M}\ud {\bf x}
\eeq
Next, to use the Lax-Milgram theorem, one only needs to prove $a_\alpha$
is coercive as all other theorem constitutive assumptions are
obviously fulfilled. 
\\
The fact that $a_\alpha$ is coercive is an immediat consequence of the
inequality
\begin{eqnarray}\label{ff11}
& & \left| \int_{\Omega}{\bf k} \varphi \cdot \nabla\left( \frac{\varphi}{M} \right) \ud {\bf x} \right| \leq \int_{\Omega} \left| {\bf k} \right| \left| \frac{\varphi}{M^{1/2}} \right| \left| M^{1/2} \nabla\left( \frac{\varphi}{M} \right) \right| \ud {\bf x}\nonumber \\
& & \leq \frac{1}{2} \left( |\varphi|_1^2 + \|{\bf k}\|^2_{L^{\infty}}\|\varphi \|^2_{L^2_M} \right)
\end{eqnarray}
and of the choice of $\alpha$.
\end{proof}
 Let then $B_\alpha:\left( H_M^1\right)'\rightarrow H_M^1$ denote the inverse operator of $L_\alpha$.
\\
Clearly 
$$
B_\alpha \in \mathcal{L} \left ( (H_M^1)', \; H_M^1 \right ).
  $$
and, also,
$$
B_\alpha \in \mathcal{L} \left ( L_M^2, \; L_M^2 \right ).
  $$
\begin{lemma}\label{est3}[\textbf{Weak Maximum Principle}]
Let $f\in \left( H_M^1 \right)',\,f\geq 0$, and  $u=B_\alpha f$.  Then
$u\geq 0$.
\end{lemma}
\begin{proof}
The proof is classical and consists to choose 
$\varphi = u^-$ in the corresponding variational formulation.
(see for exemple \cite{evans} for the non-degenerate case $M \equiv 1$).
\end{proof}

\begin{lemma}\label{est6}[\textbf{Comparison Principle}]
Let $\Omega'$ be an open set such that $\Omega' \subset \Omega$, and
let $\overline{\Omega'}$ denote its closure.  Let $u,v\in
H_M^1(\Omega)$ so  that $L_\alpha u$ and  $ L_\alpha v$
be functions well defined on $\Omega'$ .
Assume that $L_\alpha u  \geq L_\alpha v $ on $\Omega'$.
Then:

\begin{itemize}
 \item Case 1.  If $\overline{\Omega'}\subset\Omega$ and if $u\geq v$ on $\partial\Omega'$, then $u\geq v$ on $\Omega'$. 
 \item Case 2.  If $\overline{\Omega-\Omega'}\subset\Omega$ and if $u\geq v$ on $\partial(\Omega-\Omega')$, then $u\geq v$ on $\Omega'$.
\end{itemize}

\end{lemma}

\begin{proof}
Let $w=u-v\in H_M^1(\Omega)$ and $f=L_\alpha(u-v)$.  Then, for $\forall\varphi\in H_M^1(\Omega) $, such that $\varphi|_{\Omega-\Omega'}=0 $, one has:

\beq\label{est7}
\int_{\Omega'} M \nabla\left( \frac{w}{M} \right) \cdot \nabla\left( \frac{\varphi}{M} \right)\ud {\bf x}'
-\int_{\Omega'}{\bf k} w \cdot \nabla\left( \frac{\varphi}{M} \right) 
\ud {\bf x}' +\alpha\int_{\Omega'} \frac{w\varphi}{M}\ud {\bf x}'=
\int_{\Omega'}\frac{f\varphi}{M}\ud {\bf x}'
\eeq

We now take in \ref{est7} \ $ \varphi : \Omega \rightarrow \mathbb{R} $
defined by
\begin{equation}
\label{est8}
\varphi =
\left \{
\begin{array}{ccc}
w^- & \quad \text{on} & \quad \Omega' \\
0 & \quad \text{on} &\quad \Omega-\Omega'
\end{array}
\right.
\end{equation}
and we easily obtain the result.  
\end{proof}

We now introduce for any $\beta > 0$
\beq\label{sfs18}
L_{2,\beta}:=\left\lbrace  \varphi\in L^1_{\text{loc}}(\Omega)\quad\text{s.t.}\quad \frac{\varphi}{M^\beta}\in L_{2}(\Omega)   \right\rbrace.
\eeq
Actually, $L_{2,\beta}(\Omega)$ is a Hilbert space endowed with the norm 
\beq\label{sfs19}
\|\varphi\|_{L_{2,\beta}(\Omega)}=\left\| \frac{\varphi}{M^\beta} \right\|_{L_{2}(\Omega)},\,\forall \varphi\in L_{2,\beta}(\Omega).
\eeq
We have, as a straightforward consequence of Theorem 6.2.5 of
\cite{ne}, the following continuous inclusion:
\beq\label{sfs17} 
H_M^1(\Omega)\displaystyle\mathop{\subset}_{\text{cont}}L_{2, {1}/{2} + {1}/{\delta}}(\Omega)
\eeq
Next:

\begin{proposition}\label{sfs20}
 \begin{itemize}
  \item[(a)] If $\beta>1/2-1/(2\delta)$, then $X_\beta \displaystyle\mathop{\subset}_{\text{cont}}L_{M}^2$.
\item[(b)] If $\beta>1/2-3/(2\delta)$, then $X_\beta \displaystyle\mathop{\subset}_{\text{cont}}\left( H_{M}^1\right)'$.
 \end{itemize}
\end{proposition}

\begin{proof}

\begin{itemize}
 \item[(a)]  Let $\varphi\in X_\beta$, arbitrarily.  Then, $|\varphi({\bf x})|\leq M^\beta({\bf x}) \|\varphi\|_{X_\beta}$, from which we get that:
\beq\label{sfs21}
\int_{\Omega} \frac{\varphi^2}{M}\ud {\bf x}\leq \|\varphi\|_{X_\beta}^2\int_{\Omega}\frac{M^{2\beta}}{M} \ud {\bf x}=\|\varphi\|_{X_\beta}^2 \int_{\Omega}\left(1-\|{\bf x}\| ^2 \right)^{(2\beta-1)\delta} \ud {\bf x}
\eeq
However, $\int_{\Omega}\left(1-\|{\bf x}\| ^2 \right)^{(2\beta-1)\delta} \ud {\bf x}<+\infty $  iff $\beta>1/2-1/(2\delta). $

\item[(b)] Let $\varphi\in X_\beta$ and $\psi \in H_M^1$
arbitrarily. We have
\beq\label{sfs23} 
\left| \int_{\Omega} \frac{\varphi\psi}{M} \ud {\bf x} \right|\leq \|\varphi\|_{X_\beta}\int_{\Omega}M^{\beta-1}|\psi| \ud {\bf x}
\eeq
However:
\begin{equation}\label{sfs24}
\int_{\Omega}M^{\beta-1}|\psi| \ud {\bf x} 
\leq\underbrace{\left\| M^{-{1}/{2}-{1}/{\delta}} \psi\right\|_{L^2(\Omega)}}_{=\|\psi\|_{L_{2,1/2+1/\delta}}} \left\| M^{\beta-{1}/{2}+{1}/{\delta}} \right\|_{L^2(\Omega)}
\end{equation}

Moreover, the $L^2$-norm of $ M^{\beta-{1}/{2}+{1}/{\delta}} $
is finite iff $ \beta > 1/2 - 3/(2 \delta) $.
\\
Using also the continuous inclusion \eqref{sfs17} we have the result stated.
  
\end{itemize}

\end{proof}

\begin{proposition}\label{sfs27}
Let $\beta$ be such that $0\leq \beta\leq 1/2- 1/\delta$.  Then, for any $u\in H^1_M$,  
$\displaystyle\frac{u}{M^\beta}\in H^1_0(\Omega)$; moreover, $\left\| \displaystyle\frac{u}{M^\beta} \right\|_{H^1}\leq c \left \|u \right\|_{H^1_M}$.   

\end{proposition}

\begin{proof}
 
Let $v=\displaystyle \frac{u}{M^\beta}=\frac{u}{M^{1/2}}
M^{1/2-\beta}$.  We actually need to prove that $\displaystyle
\frac{u}{M^{1/2}}\in H^1(\Omega)$.  From this, since $M^{1/2-\beta}\in
\mathscr{C}^1(\overline{\Omega})$ 
and $M^{1/2-\beta}|_{\partial\Omega}
=0$, it will follow that $v \in {H^1_0}(\Omega)$.  

To begin with, notice first that $\displaystyle \frac{u}{M^{1/2}}
\in L^2(\Omega)$, as $u\in L^2_M(\Omega)$.  Next, 

\beq\label{sfs28}
\nabla\left( \frac{u}{M^{1/2}} \right)=\nabla\left( M^{1/2}\frac{u}{M}
\right)=M^{1/2}\nabla\left( \frac{u}{M} \right)+\nabla\left( M^{1/2}
\right)\frac{u} M.
\eeq

Now, $\displaystyle M^{1/2}\nabla\left( \frac{u}{M} \right)\in
L^2(\Omega) $ as $u\in H^1_M $.  Let us next show that $\displaystyle
\nabla\left( M^{1/2} \right)\frac{u} M \in L^2(\Omega) $.  One has:

\beq\label{sfs29}
\nabla\left( M^{1/2} \right)\frac{u} M=\frac{\nabla M}{2M^{3/2}}u=\frac{u}{M^{1/2+1/\delta}}\underbrace{\frac{\nabla M}{2M^{1-1/\delta}}}_{\in L^\infty(\Omega)}.
\eeq

Next, by \eqref{sfs17} $\displaystyle \frac{u}{M^{1/2+1/\delta}}
\in L^2(\Omega)$, from which we infer that $\displaystyle 
\frac{u}{M^{1/2}}\in H^1(\Omega) $, and further on that 
$\displaystyle \frac{u}{M^{\beta}}\in H^1_0(\Omega) $.  
It is easily deduced that:

\beq\label{sfs30}
\left\| \frac{u}{M^\beta} \right\|_{H^1_0}\leq c \left \|u \right\|_{H^1_M}.
\eeq

\end{proof}

\begin{remark}
Taking $\beta=0$ in Proposition \ref{sfs27} we deduce that
$u\in H^1_0(\Omega)$ whenever $u\in H^1_M(\Omega) $, which triggers that the trace of $u$ on the boundary $\partial\Omega$ is equal to zero.  
\end{remark}

\begin{proposition}\label{sfs31}
Let $\varphi\in X_\beta$, $\beta>1/2+1/(2\delta)$, be such that 
$\nabla \varphi \in (X_\gamma)^n$ with 
\\
 $ \gamma>1/2-1/(2\delta)$.  
Then $\varphi\in H_M^1$.
\end{proposition}

\begin{proof}
From Proposition \ref{sfs20} a) we have
$\varphi\in L_M^2$.  Next,

\begin{eqnarray}\label{sfs32}
\int_\Omega M \left| \nabla \left ( \frac{\varphi}{M} \right )
 \right|^2\ud {\bf x} & = & \int_\Omega M\left| \frac{1}{M}\nabla\varphi-\frac{\nabla M}{M^2}\varphi \right|^2\ud {\bf x} \nonumber\\
& \leq & 2\int_\Omega  \frac{1}{M}\left|\nabla\varphi\right|^2 \ud {\bf x} +2 \int_\Omega  \frac{1}{M^3}\left|\nabla M\right|^2 \varphi^2\ud {\bf x} 
\end{eqnarray}

However,

\beq\label{sfs33}
\int_\Omega \frac 1 M | \nabla \varphi|^2 \ud {\bf x} \leq c
\int_\Omega \frac{1}{M}M^{2\gamma}\ud {\bf x} < + \infty
\eeq

provided that $\delta(2\gamma-1)>-1$, which amounts to $\gamma>1/2-1/(2\delta) $.  

Next,

\begin{eqnarray}\label{sfs34}
\int_\Omega  \frac{\left|\nabla M\right|^2 \varphi^2}{M^3} \ud {\bf x} 
\leq c_1  \int_\Omega  \frac{\left|\nabla M\right|^2 M^{2\beta}}{M^3} 
\ud {\bf x}  & \leq & c_2  \int_\Omega  \frac{\left(1-\left\|{\bf x}
\right\|^2\right)^{2(\delta-1)} }{\left(1-\left\|{\bf x}\right\|^2
\right)^{3\delta}}\left(1-\left\|{\bf
    x}\right\|^2\right)^{2\delta\beta} 
\ud {\bf x} \nonumber\\
& \leq & c_2 \int_\Omega\left(1-\left\|{\bf x}\right\|^2
\right)^{2\delta\beta+2\delta-2-3\delta} \ud {\bf x}
\end{eqnarray}

For the above equation \eqref{sfs34} to hold true it is necessary that $2\delta\beta-\delta-2>-1$, i.e. $\beta>1/2+1/(2\delta)$.  

\end{proof}
Since $M^\beta \in X_\beta$ for any $\beta \geq 0$ we have the
following direct consequence 
 of the above result:

\begin{proposition}\label{sfs35}
For any  $\beta>1/2+1/(2\delta)$, we have  $M^\beta\in H_M^1$.
\end{proposition}

\subsection{Continuity and compactness properties of  $B_\alpha$}

The goal now is to appropriately introduce several $(Y', Y'')$ paires of
Banach spaces such that $B_\alpha$ is well defined and
continuous from $Y'$ to $Y''$.  Some compactness
properties of $B_\alpha$, needed further on, are also proved.

\begin{lemma}\label{p1}
(i) Let $\beta_2\in\mathbb{R}$ such that $1/2-3/(2\delta)<\beta_2
\leq 1/2-1/\delta$. Then  $B_\alpha\in
\mathcal{L}(L_M^2,X_{\beta_2})$.  
Moreover, $B_\alpha $ is a compact application from $L_M^2$ onto $X_{\beta_2}$.

(ii) Let $\beta_1,\beta_2\in \mathbb{R}$ such that
$1/2-3/(2\delta)<\beta_2\leq 1/2-1/\delta$, and $\beta_1\geq\beta_2$.
Then $B_\alpha\in \mathcal{L}(X_{\beta_1},X_{\beta_2})$.  
Moreover, $B_\alpha $ is a compact application from $X_{\beta_1}$ 
onto $X_{\beta_2}$.
 
\end{lemma}

\begin{proof}
The proofs for the above two statements are pretty much similar in
nature; henceforth, we offer below a global proof, and pause wherever necessary to
particularize it so to get the results in either (i) or (ii).  
Keeping that in mind, let $f\in L_M^2$ (for the (i) part) and 
$f\in X_{\beta_1}$  (for the (ii) part) and let $u=B_\alpha f$.  Observe that:

\beq\label{p2} L_M^2\mathop{\subset}_{\text{cont.}}(H_M^1)'
\eeq

and that

\beq\label{p3}
X_{\beta_1}\mathop{\subset}_{\text{cont.}}(H_M^1)'
\eeq

as consequences of Proposition \ref{sfs20}.
In both cases $f\in (H_M^1)'$
and $u \in H_M^1$ solves the equation
\beq\label{p4}
-\nabla\cdot\left[M\nabla\left(\frac{u}{M} \right)\right]+\nabla\cdot({\bf k}u) +\alpha u=f,\, u\in H_M^1
\eeq
We also have
$$
\| u \|_{H_M^1} \leq c \| f \|_{L_M^2} \quad \text{ for the part (i)}
  $$
and
$$
\| u \|_{H_M^1} \leq c \| f \|_{X_{\beta_1}} \quad \text{ for the part (ii)}.
  $$
Denote $v({\bf x})=u({\bf x})/M^{\beta_2}({\bf x}) $; we first take 
on to prove that $v$ is bounded on $\Omega$, which prompts
that $u$ belongs to $X_{\beta_2}$. 
 
Making use of the fact that $u=v M^{\beta_2}$ into \eqref{p4} leads to:

\begin{eqnarray}\label{p5}
& & -\nabla\cdot\left[M^{\beta_2}(\nabla v)+(\beta_2-1)M^{\beta_2-1}(\nabla M) v \right]+(\nabla\cdot {\bf k})M^{\beta_2}v \nonumber\\
& & +\beta_2 M^{\beta_2-1}(\nabla M)\cdot{\bf k}v+{\bf k}M^{\beta_2}\cdot(\nabla v)+\alpha M^{\beta_2}v=f,\, \forall{\bf x}\in\Omega
\end{eqnarray}

which, after a few re-arrangements, can be re-written as:

\beq\label{p6}
-\triangle v=g
\eeq

where
\begin{eqnarray}\label{p7}
g & = &  \frac{f}{M^{\beta_2}}+\left[(2\beta_2-1)\frac{\nabla M}{M}-{\bf k} \right]\cdot \nabla v \nonumber\\
& + & \left[(\beta_2-1)\frac{\triangle M}{M}+(\beta_2-1)^2 \frac{|\nabla M|^2}{M^2}-\nabla\cdot{\bf k}-\beta_2\frac{\nabla M}{M}\cdot{\bf k}-\alpha \right]v 
\end{eqnarray}
We also deduce from Proposition \ref{sfs27} that
\beq\label{v-in-H01}
v \in H_0^1(\Omega)
\eeq
In the following we shall obtain some convenient estimates
for the function $g$. We have
\begin{eqnarray}\label{p8}
\nabla v & = & \nabla\left(\frac{u}{M}M^{1-\beta_2} \right)= M^{1-\beta_2}\nabla\left(\frac{u}{M}\right)+\nabla(M^{1-\beta_2})\frac{u}{M} \nonumber\\
& = & M^{1/2-\beta_2}\underbrace{M^{1/2}\nabla\left(\frac{u}{M}\right)}_{\in L^2(\Omega)\,\text{as} \,u\in H_M^1}+(1-\beta_2)M^{-\beta_2-1}(\nabla M)u.
\end{eqnarray}
Using equation \eqref{p7} we get:
\beq\label{p9}
g=\frac{f}{M^{\beta_2}}+M^{1/2} g_1({\bf x}) \cdot
\nabla\left(\frac{u}{M}\right)+g_2({\bf x})u
\eeq

where, in the above,
\begin{eqnarray}
g_1({\bf x}) & = & \left[(2\beta_2-1)\frac{\nabla M}{M} -{\bf k}\right] M^{1/2-\beta_2} \nonumber\\
g_2({\bf x}) & = & (1-\beta_2)\left[(2\beta_2-1)\frac{\nabla M}{M} -{\bf
    k}\right] \cdot
\frac{\nabla M}{M^{\beta_2+1}} \nonumber\\
& + & \frac{1}{M^{\beta_2}}\left[(\beta_2-1)\frac{\triangle
    M}{M}+(\beta_2-1)\frac{|\nabla M|^2}{M^2}-\nabla\cdot{\bf
    k}-\beta_2\frac{\nabla M}{M}
\cdot {\bf k}-\alpha \right] \nonumber
\end{eqnarray}

For the (i) part of Lemma \ref{p1} one has:

\beq\label{p10}
\left\|\frac{f}{M^{\beta_2}} \right\|_{L^2}\leq\left\|M^{1/2-\beta_2}\right\|_{L^\infty}\left\|f\right\|_{L^2_M}
\eeq
while for the (ii) part of Lemma \ref{p1} one gets:

\beq\label{p12}
\left\|\frac{f}{M^{\beta_2}} \right\|_{L^2}\leq \left\|M^{\beta_1-\beta_2}\right\|_{L^\infty}\left\|f\right\|_{X_{\beta_1}}
\eeq
Moreover, 
$$
\frac{\nabla M}{M} \mathop{\sim}_{\|{\bf x}\|\to 1}
\frac{1}{1-\|{\bf x}\|^2}=\frac{1}{M^{1/\delta}}.
  $$
Therefore, the above leads to
$g_1\in L^{\infty}(\Omega)$. We then deduce $g_1 M^{1/2}\nabla(u/M)
\in L^2(\Omega)$ and 
\beq
\label{ineq-g1}
\left\|g_1 M^{1/2}\nabla(u/M) \right\|_{L^2(\Omega)}\leq
c_1\|u\|_{H_M^1}.
\eeq
Now, observe that: 

$$
g_2({\bf x})\mathop{\sim}_{\|{\bf x}\|\to 1}\frac{1}{M^{\beta_2+2/\delta}}
  $$

which implies

$$
g_2({\bf x})u\mathop{\sim}_{\|{\bf x}\|\to 1}\frac{u}
{M^{1/2+1/\delta}}\mathop{\underbrace{M^{1/2-\beta_2-1/\delta}}}_{\in 
L^\infty,\,\text{as}\, \beta_2\leq 1/2-1\delta}.
  $$

We deduce with the help of inclusion \eqref{sfs17} that
\beq\label{p19}
\|g_2 u\|_{L^2(\Omega)}\leq c_1\|u\|_{H_M^1}
\eeq
and further on, from \eqref{p9}, \eqref{p10}, \eqref{p12}, \eqref{ineq-g1} and
\eqref{p19}, that
\beq\label{p21}
\|g \|_{L^2(\Omega)}\leq c\|f\|_{L_M^2}
\eeq
for part (i), and
\beq\label{p22}
\|g \|_{L^2(\Omega)}\leq c\|f\|_{X_{\beta_1}}
\eeq
for (ii) part.
\\
Now since $v$ satisfies \eqref{p7} and \eqref{v-in-H01}
we obtain 
$v\in H^2(\Omega)$, and $\|v\|_{H^2(\Omega)}\leq
c\|g\|_{L^2(\Omega)}$.  
By the Sobolev's inclusion $H^2(\Omega)\displaystyle
\mathop{\subset}_{\text{cont}} \mathcal{C}(\overline{\Omega}),
\, n=2,3 $ it follows that $v\in \mathcal{C}(\overline{\Omega})$, and
that $\|v\|_{\mathcal{C}(\overline{\Omega})}\leq
c_1\|g\|_{L^2(\Omega)}$.  
\\
Thus $ u \in X_{\beta_2} $ and
\beq\label{p24}
\|u\|_{X_{\beta_2}}\leq c_1\|g\|_{L^2(\Omega)}
\eeq
Next, making use of \eqref{p21} and \eqref{p22}, we deduce, for part 
(i) that $B_\alpha\in \mathscr{L}(L_M^2,X_{\beta_2})$, and for 
part (ii) that $B_\alpha\in \mathscr{L}(X_{\beta_1},X_{\beta_2})$.

In order to show the compactness of $B_\alpha$, 
let $(f_q)_{q\in \mathbb{N}}$ be a bounded sequence in $L_M^2$ for part (i), 
and in $X_{\beta_1} $  for part (ii), respectively.  
Denote $u_q=B_\alpha(f_q)\in X_{\beta_2}$, and $v_q=u_q/M^{\beta_2}$. 
Next it is proved  that $v_q$ is bounded in $H^2(\Omega)$. 
  As the domain $\Omega$ is bounded, the inclusion $H^2(\Omega)
\subset \mathcal{C}(\overline{\Omega})$ is compact; hence there 
exists a subsequence $q'$ of $q$ and a
$v\in\mathcal{C}(\overline{\Omega})$ such that
$v_{q'}\displaystyle\mathop{\to}_{\mathcal{C}
(\overline{\Omega})}v$.   Denoting $u=vM^{\beta_2}$, we have that 
$u\in X_{\beta_2}$ and $\displaystyle\mathop{\sup}_{{\bf x}\in 
\Omega} \frac{|v_{q'}({\bf x})-v({\bf x})|}{M^{\beta_2}({\bf x})} 
\displaystyle\mathop{\to0}_{q'\to+\infty}$.  Therefore $\displaystyle
\mathop{u_{q'}\to u}_{q'\to+\infty}$ with respect to the 
$X_{\beta_2}$ space topology.  
 
\end{proof}

For any $r>0$, let us denote
$\Omega_r:=\{{\bf x} \,:\, \|{\bf x}\|<r\}  \equiv B(0,r)$.

\begin{lemma}\label{p25}
Let $\beta$ be such that $1/2+1/(2\delta)<\beta<1$.  Then 
 $B_\alpha\in\mathscr{L}(X_{\beta-2/\delta},X_\beta) $.  
\end{lemma}

\begin{proof}\label{p26}

Proposition \ref{sfs20} b)gives that $X_{\beta-2/\delta}\in (H_M^1)'$ 
(as $1/2-3/(2\delta)<\beta-2/\delta$), which entails that the 
operator $B_\alpha$ is well defined over $X_{\beta-2/\delta}$.
Let $f\in X_{\beta-2/\delta}$ and $u=B_\alpha(f)$.  
We have to prove the validity of the following assertion:
\\
There exists $A' > 0$ independent on $f$ such that
\beq\label{p27}
AM^\beta ({\bf x})-u({\bf x})\geq0, \quad \forall {\bf x}\in \Omega
\eeq
and
\beq\label{p28}
AM^\beta ({\bf x})+u({\bf x})\geq0, \quad \forall {\bf x}\in \Omega
\eeq

where we denoted $ A = A' \| f \|_{\beta - 2/\delta} $.
We shall provide a proof for the first one only, i.e. for \eqref{p27}, 
as the other may be proved similarly.  
The proof for \eqref{p27} relies on the
Comparison Principle stated in Lemma \ref{est6}.

One has:
\begin{eqnarray}\label{p29}
 L_\alpha(M^\beta) & = & M^\beta(\alpha+\nabla\cdot{\bf k})+M^{\beta-1}[\beta{\bf k}\cdot\nabla M+(1-\beta)\Delta M] \nonumber\\
& - & M^{\beta-2}(1-\beta)^2|\nabla M|^2
\end{eqnarray}

As:
\begin{eqnarray}\label{p30}
\nabla M & = & -2 \delta {\bf x}(1-\|{\bf x}\|^2)^{\delta-1} 
\nonumber\\
\Delta M & = & -2n\delta(1-\|{\bf x}\|^2)^{\delta-1}+4\delta
(\delta-1)\|{\bf x}\|^2 (1-\|{\bf x}\|^2)^{\delta-2} \nonumber
\end{eqnarray}

then,
\begin{eqnarray}
 L_\alpha(M^\beta) & = & a_0({\bf x})(1-\|{\bf
   x}\|^2)^{\delta\beta-2}+a_1({\bf x})(1-\|{\bf
   x}\|^2)^{\delta\beta-1}+a_2({\bf x})(1-\|{\bf
   x}\|^2)^{\delta\beta},\,\text{where} \nonumber \\
a_0({\bf x}) & = & 4\delta(1-\beta)(\delta\beta-1)\|{\bf x}\|^2 
\nonumber\\
a_1({\bf x}) & = & -[2\delta\beta{\bf x} \cdot {\bf
  k}+2(1-\beta)n\delta]
\nonumber\\
a_2({\bf x}) & = & \alpha+\nabla\cdot k({\bf x})
\nonumber 
\end{eqnarray}
It is clear that \ $a_0({\bf x}) \geq 0, \; \forall \; {\bf x} \in 
\Omega$.
Next, since $f\in X_{\beta-2/\delta}$, we deduce
\begin{equation}
\nonumber
-f({\bf x})\geq-\|f\|_{X_{\beta-2/\delta}}M^{\beta-2/\delta}({\bf x})=-\|f\|_{X_{\beta-2/\delta}}(1-\|{\bf x}\|^2)^{\delta\beta-2}
\end{equation}

Then:
\begin{eqnarray}
 A L_\alpha(M^\beta)-f & \geq & [A a_0({\bf x})-\|f\|_
{X_{\beta-2/\delta}}](1-\|{\bf x}\|^2)^{\delta\beta-2}+Aa_1({\bf x})
(1-\|{\bf x}\|^2)^{\delta\beta-1} \nonumber \\
 & + & Aa_2({\bf x})(1-\|{\bf x}\|^2)^{\delta\beta} \nonumber
\end{eqnarray}

In the following, we restrict ourselves to $\Omega-\Omega_{1/2}$, 
henceforth $\|{\bf x}\|\geq 1/2$.  Then 
$$
a_0({\bf x})\geq a_0^0:=\delta(1-\beta)(\delta\beta-1)>0,\quad \forall
{\bf x}\in\Omega-\Omega_{1/2}
  $$
and 
\beq
\label{p33}
\begin{split}
 A L_\alpha(M^\beta)-f \geq & A a_0^0(1-\|{\bf x}\|^2
)^{\delta\beta-2}  [ 1-\frac{\|f\|_{X_{\beta-2/\delta}}}{A a_0^0}
+\frac{a_1({\bf x})}{a_0^0}(1-\|{\bf x}\|^2) 
\\
+ & \frac{a_2({\bf x})}{a_0^0}(1-\|{\bf x}\|^2)^2  ],
\quad \forall
{\bf x}\in\Omega-\Omega_{1/2}
\end{split}
\eeq
Assume $f\neq0$ (this is not too restrictive as, whenever $f=0$,
the inequality \eqref{p27} is satisfied with $A=0$).
\\
Let us choose $ r_0 \in \left [ \frac 1 2, 1 \right [ $ close
enough to 1 such that
\beq
\label{def-r0}
\begin{cases}
\dfrac {1 - r_0^2} {a_0^0} \displaystyle\mathop{\sup}_{{\bf x} \in \Omega} |a_1({\bf x})|
\leq \dfrac 1 4
\\
\dfrac {(1 - r_0^2)^2} {a_0^0} \displaystyle\mathop{\sup}_{{\bf x} \in \Omega} |a_2({\bf x})|
\leq \dfrac 1 4
\end{cases}
\eeq
On the other hand, from Lemma \ref{p1} (ii) with $ \beta_1 =
\beta - 2/\delta $ and $ \beta_2 = \min \{\beta -  2/\delta , 1/2 -
1/\delta\} $ we deduce $ u \in X_{\beta_2} $ (since $f \in
X_{\beta_1}$) and
\beq\label{p37}
\|u\|_{\mathcal{C}(\overline{\Omega})}
\leq c_1\|u\|_{X_{\beta_2}}\leq c_2 \|f\|_{X_{\beta-2/\delta}}
\eeq

Take now
$$
A = \max \left\{ \frac 4 {a_0^0}, \, \frac {c_2} {(1 - r_0^2)^{\delta
    \beta}} \right\} \| f \|_{X_{\beta - 2/\delta}}.
  $$
Clearly, from (\ref{p33}), (\ref{def-r0})
 and (\ref{p37}),   
$$
L_\alpha (A M^\beta - u) = A L_\alpha (M^\beta) - f \geq 0
\quad \text{ on } \; \Omega - \Omega_{r_0}
  $$
and
$$
A M^\beta \geq  u \quad  \text{ on } \; {\overline \Omega_{r_0}}.
  $$
respectively.  Invoking the Comparison Principle (Lemma \ref{est6}) and the fact 
that $ u, M^\beta \in H_M^1 $, leads to
$$
A M^\beta \geq u \quad \text { on } \; \Omega - \Omega_{r_0}, \quad
\text { which implies } \quad
A M^\beta \geq u \quad \text { on } \; \Omega.
  $$
This ends the proof.
 
\end{proof}

\begin{lemma}\label{p41}
$B_\alpha\in\mathscr{L}(X_{1-1/\delta},X_1)$
\end{lemma}

\begin{proof}
 Let $f\in X_{1-1/\delta}$ and $u=B_\alpha(f)$.   Let $W:\Omega
\to\mathbb{R},\, W({\bf x})=e^{\lambda\|{\bf x}\|^2} $.  
The job is now to prove that there exists $\lambda>0$ and $A'>0$ 
independent on $f$, such that
$$
|u({\bf x})|\leq A W({\bf x})M({\bf x}),\quad \forall {\bf x}\in \Omega
  $$
where we denoted $ A = A' \| f \|_{X_{1-1/\delta}}.$
  Actually we take on to prove $AWM\geq u$ only, as 
$AWM\geq -u$ can be proved similarly.  To achieve this, we again
make use of the Comparison Principle. 

One has:
\beq
\nonumber
L_\alpha(MW)=M[-\Delta W+{\bf k}\cdot\nabla W+(\alpha+\nabla\cdot{\bf k})W]+\nabla M\cdot({\bf k}W-\nabla W)
\eeq

However, $\nabla W=2\lambda{\bf x}W$, $\Delta W=(2\lambda n+\lambda^2\|{\bf x}\|^2)W$, hence:

\begin{eqnarray}\label{p42}
 L_\alpha(MW) & = & MW(-4\lambda^2\|{\bf x}\|^2-\lambda n+2{\bf k}\cdot\lambda{\bf x}+\alpha+\nabla\cdot{\bf k}) \nonumber\\
& + & 2\delta M^{1-1/\delta}W(2\lambda\|{\bf x}\|^2-{\bf k}\cdot{\bf x})
\end{eqnarray}

Let us take $\lambda = \lambda_0$  with $\lambda_0$ given in
\eqref{lamb0}. We obtain 
$$
2\lambda\|{\bf x}\|^2-{\bf k}\cdot{\bf x} \geq 1.
  $$
From hypothesis
\eqref{hyp-alp} on $\alpha$ we obtain
\beq
\nonumber
-4\lambda^2 \|{\bf x}\|^2-\lambda n+2 \lambda {\bf k}\cdot
{\bf x}+\alpha+\nabla\cdot{\bf k}\geq 0 \quad \text{ on }
\Omega-\Omega_{1/2}
\eeq
which gives
\beq\label{p46}
L_\alpha(MW)\geq 2\delta WM^{1-1/\delta},\, \forall {\bf x}\in
\Omega-\Omega_{1/2}.
\eeq

Next, as $f\in X_{1-1/\delta}$, one gets
\beq
\nonumber
-f({\bf x})\geq -\|f\|_{X_{1-1/\delta}}M^{1-1/\delta}({\bf x}),\, \forall{\bf x}\in \Omega
\eeq

and invoking further on \eqref{p46} leads to:

\beq\label{p48}
AL_\alpha(MW)-f\geq [2\delta WA-\|f\|_{X_{1-1/\delta}}]M^{1-1/\delta},
\;\forall {\bf x}\in\Omega-\Omega_{1/2}, \; \forall A > 0.
\eeq

Choose $A>0$ such that $2\delta A W({\bf x})\geq \|f\|_{X_{1-1/\delta}},\,\forall {\bf x}\in\Omega-\Omega_{1/2}$.  For instance, any $A$ such that:

\beq\label{p49}
A\geq \frac{1}{2\delta}\|f\|_{X_{1-1/\delta}}
\eeq 

will fit in.  Then:

\beq\label{p50}
L_\alpha(AMW-u)\geq0 \quad \quad \text { on } \quad \Omega - \Omega_{1/2}
\eeq

On the other hand, one needs to choose $A$ so that $AMW\geq u$ holds
true over $\overline{\Omega_{1/2}}$.  We proceed as in the
proof of Lemma \ref{p25}. A sound choice for $A$ is one such that 

\beq\label{p51}
A\left( \mathop{\min}_{{\bf x}\in\Omega_{1/2}}W({\bf x})\right)  \left(\mathop{\min}_{{\bf x}\in\Omega_{1/2}}M({\bf x})\right)\geq \mathop{\max}_{{\bf x}\in\Omega}u({\bf x})
\eeq

Next, $\displaystyle \mathop{\min}_{{\bf x}\in\Omega_{1/2}}W({\bf
  x})=1$, 
$\displaystyle \mathop{\min}_{{\bf x}\in\Omega_{1/2}}M({\bf x})=
(3/4)^{\delta}$, and we are left over to inquire about $\displaystyle
\mathop{\max}_{{\bf x}\in\Omega}u({\bf x})$.  To get an answer to, we
shall 
call in Lemma \ref{p1} with 
$\beta_1= 1 - 1/\delta $ and $ \beta_2=1/2-1/\delta$.
One has
$$
\|u\|_{\mathcal{C}(\overline{\Omega})}\leq c_1\|u\|_{X_{1/2-1/\delta}}
\leq c_2 \|f\|_{X_{1-1/\delta}}.
  $$
Then, one may choose 
$A\geq (4/3)^\delta c_2 \|f\|_{X_{1-1/\delta}}$ to ensure 
\eqref{p51} holds true.  Finally, taking into account \eqref{p49}, 
we are left to choose 
\beq
\nonumber
A=\max \{1/(2\delta),(4/3)^\delta c_2\}\|f\|_{X_{1-1/\delta}}
\eeq
and we end the proof exactly as in Lemma \ref{p25}, taking into
account the fact that $M W \in H_M^1$, so that the Comparison Principle
can be made use of.

\end{proof}

\subsection{Strong Maximum Principle for the $B_\alpha$ operator}

This section aim is to prove the following ``Strong Maximum
Principle `` property for $B_\alpha$:
for any $f \in P_1 - \{0\}$, 
$B_\alpha f \in \stackrel{\circ}{P}_1$.  
\\

The following  weaker result is first proved. 
\begin{lemma}\label{m1}
 Let $f\in P_1,\, f\neq0$ and  $u=B_\alpha f$. 
Then $u({\bf x})>0,\,\forall {\bf x}\in \Omega$.
\end{lemma}

\begin{proof}
We adapt here the classical proof for the case where $M$ is
equal to 1 (the non-degenerate case; see for example Gilbarg and Trudinger 
\cite{gil-tru} or Evans \cite{evans}). 

We remark first that $u$ is continuous on $\Omega$.
Assume $\exists{\bf x}\in \Omega$ such that $u({\bf x})=0$.  

Denote $V_e:=\{{\bf x}\in\Omega,\,u({\bf x})=0 \}$, 
$V_s:=\{{\bf x}\in\Omega,\,u({\bf x})>0 \}$, 
 $V_e\cup V_s=\Omega$. 

By hypothesis $V_e\neq \emptyset$, as well as 
$V_s\neq \emptyset$. It is clear that  $V_s$ is open 
and that $\partial V_e\nsubseteq\partial\Omega$.  Let ${\bf z}_0\in 
\partial V_e\cap \Omega \neq \emptyset$; then $u({\bf z}_0)={\bf
  0}$.  Denote $d=\displaystyle\mathop{\inf}_{{\bf
    z}\in\partial\Omega}|
{\bf z}_0-{\bf z}|>0$, thus $|{\bf z}_0|=1-d$.
Let $r_1\in ]0,d/4[$ be small enough, and fix $ {\bf x}_0 \in 
V_s $ such that
$|{\bf x}_0-{\bf z}_0|<r_1$.  As $V_s$ is an open subset, 
there exists $r_2>0$ such that $B({\bf x}_0,r_2)\subset V_s$.  
Therefore choose $r_0=\sup \{r\quad\text{s.t.}\quad  B({\bf
  x}_0,r)\subset V_s\}$.  Then there exists ${\bf y}_0 \in
\overline{B({\bf x}_0,r_0)}\cap V_e\neq\emptyset$, $u({\bf y}_0)=0$,
and $|{\bf y}_0-{\bf x}_0|=r_0$.  This prompts $r_0\leq r_1$, hence one
may choose a small enough $r_0$.  Thus, $u({\bf y}_0)=0,\, 
u({\bf x})>0,\, \forall {\bf x}\in B({\bf x}_0,r_0)$.  Let the 
function $w$ be such that $w:B({\bf x}_0,r_0)\to\mathbb{R},\,
w({\bf x})=e^{-\lambda \|{\bf x}-{\bf x}_0\|^2}-e^{-\lambda r_0^2}$,
where $\lambda>0$ will be later chosen conveniently.  Denote also 
by $w$ the continuous extention of $w$ at $0$ on $\Omega$ 
 and note that $w\in H_M^1(\Omega)$, $w({\bf x})|_{{\bf x}\in\partial B({\bf x}_0,r_0)}=0$, and $w({\bf x})|_{{\bf x}\in B({\bf x}_0,r_0)}>0$.   

Next, we take on to prove that $\exists A>0$ such that $u({\bf x})
\geq Aw({\bf x}),\,\forall {\bf x}\in B({\bf x}_0,r_0)-
B({\bf x}_0,r_0/2)$.  To achieve this we shall make use of the 
Comparison Principle.  We actually evaluate $L_\alpha(u-Aw)=
f-AL_\alpha(w),\,f\geq0$, and  prove that $L_\alpha(w)\leq0$ for 
any $\forall {\bf x}\in B({\bf x}_0,r_0)-B({\bf x}_0,r_0/2)$.  
\\

Basic calculations lead to:

\beq\label{m2}
L_\alpha(w)= E_1 + E_2  
\eeq 
where, in the above
\begin{eqnarray}
\nonumber
 & & E_1=-\Delta w+\left(\frac{\nabla M}{M}+{\bf k} \right)\cdot \nabla w \\
& & E_2=\left[ \nabla\cdot\frac{\nabla M}{M}+
\nabla \cdot {\bf k}+\alpha  \right] w  
=\left[ \frac{\Delta M}{M} - \frac{|\nabla M|^2}{M^2}+\nabla\cdot {\bf
    k}+\alpha \right]w
\nonumber
\end{eqnarray}
Using the expression of $w$ we find
\begin{eqnarray}
\label{m3}
 E_1=\bigg[ & - & 4\lambda^2|{\bf x}-{\bf x}_0|^2+2\lambda
 n+\frac{4\delta\lambda {\bf x} \cdot ({\bf x}-{\bf x}_0)}
{1-\|{\bf x}\|^2} \nonumber
 \\
& - & 2\lambda{\bf k} \cdot ({\bf x}-{\bf x}_0)\bigg]
e^{-\lambda|{\bf x}-{\bf x}_0|^2 }
\end{eqnarray}
and
\begin{eqnarray}\label{m4}
E_2=\bigg[& - & \frac{2n\delta}{1-\|{\bf x}\|^2}+\frac{4\delta(\delta-1)\|{\bf x}\|^2}{(1-\|{\bf x}\|^2)^2}-\frac{4\delta^2\|{\bf x}\|^2}{(1-\|{\bf x}\|^2)^2} \\
& + & \nabla\cdot {\bf k} +\alpha\bigg](e^{-\lambda|{\bf x}-{\bf x}_0|^2 }-e^{-\lambda r_0^2})
\end{eqnarray}

Next, observe that $\|{\bf x}_0\|\leq\| \leq 1-d+r_1$,
so for any ${\bf x}\in B({\bf x}_0,r_0)$, one has 
$\|{\bf x}\|\leq  1-d+r_1+r_0\leq 1-d/2$.  
Therefore $1-\|{\bf x}\|^2\geq d(1-d/4)>0$.  Denote $d_0=1/(d-d^2/4)>0$; hence 
\beq\label{m5}
\frac{1}{1-\|{\bf x}\|^2}\leq d_0,\quad \forall {\bf x}\in B({\bf x}_0,r_0)
\eeq
Then
$$
E_1\leq \left[-\lambda^2 r_0^2+2\lambda n +4\delta\lambda r_0 d_0+2\lambda
\|{\bf k}\|_{L^\infty}r_0 \right]e^{-\lambda|{\bf x}-{\bf x}_0|^2 },
\quad \forall {\bf x} \; \text { such that } \; 
\frac {r_0} 2 \leq \| {\bf x} \| \leq r_0.
  $$
We also have
$$
E_2\leq|E_2|\leq \left[2n\delta d_0+4\delta(\delta-1)d_0^2+4\delta^2 d_0^2+\|\nabla\cdot{\bf k}\|_{L^\infty}+\alpha\right]e^{-\lambda|{\bf x}-{\bf x}_0|^2 }.  
  $$
which implies
\beq
\begin{split}
L_\alpha(w) & \leq \left[-\lambda^2 r_0^2 +\lambda(2n+4\delta r_0 d_0+2\|{\bf
  k}\|_{L^\infty}r_0)+2n\delta
d_0+(8\delta^2-4\delta)d_0^2+\|\nabla\cdot{\bf
  k}\|_{L^\infty}+\alpha\right]e^{-\lambda|{\bf x}-{\bf x}_0|^2},
\\
& \quad \forall {\bf x} \; \text { such that } \; 
\frac {r_0} 2 \leq \| {\bf x} \| \leq r_0.
\end{split}
\eeq
Then one may choose a $\lambda>0$ large enough (with $\lambda$ depending on 
$z_0$ 	and  $r_0$) so that the rhs be negative, i.e. $L_\alpha(w)\leq0$.  
Therefore 
\beq \label{m6}
L_\alpha(u-Aw)\geq0, \quad 
\forall {\bf x}\in B({\bf x}_0,r_0)-B({\bf x}_0,r_0/2), \quad \forall A>0.
\eeq
Next, as $u({\bf x})>0$ in $B({\bf x}_0,r_0)$ and $u$ is continuous in
$\Omega$ , one has $\displaystyle\mathop{\inf}_{{\bf x}\in \partial
  B({\bf x}_0,r_0/2)}u({\bf x})>0$.  Choose $A>0$ such that 
$\displaystyle\mathop{\inf}_{{\bf x}\in \partial B({\bf x}_0,r_0/2)}
u({\bf x})\geq A \displaystyle\mathop{\sup}_{{\bf x}\in \partial 
B({\bf x}_0,r_0/2)}w({\bf x})=A \left (
e^{-\lambda r_0^2/4}-e^{-\lambda r_0^2} \right )$.  
Then choose $A=\left[ \displaystyle\mathop{\inf}_{{\bf x}\in 
\partial B({\bf x}_0,r_0/2)}u({\bf x})\right]/\left 
[e^{-\lambda r_0^2/4}-e^{-\lambda r_0^2} \right]$, to get 
\beq
\nonumber
u\geq A w \quad \text { on } \quad \partial B({\bf x}_0,r_0/2).
\eeq
We also have
\beq 
\nonumber
u\geq A w \quad \text { on } \quad \partial B({\bf x}_0,r_0)
\eeq 
Then the inequality \eqref{m6} and the Comparison Principle give

\beq\label{m8}
u\geq A w, \quad \forall{\bf x}\in B({\bf x}_0,r_0)-B({\bf x}_0,r_0/2).
\eeq 
Next, the interior regularity property gives 
$u\in \mathcal{C}^1(\Omega)$.
Let  
$ \nu = \displaystyle\frac{{\bf y}_0-{\bf x}_0}{r_0}$ denote
the outward normal vector at ${\bf y}_0\in B({\bf x}_0,r_0)$.  
Then $ \displaystyle\frac{\partial u}{\partial\nu}({\bf y}_0)=
-\frac{1}{r_0}\displaystyle\mathop{\lim}_{\substack {s\to0 \\ s>0}}
\frac{1}{s×}u\left[{\bf y}_0-s({\bf y}_0-{\bf x}_0) × \right]× $.  
\\
With the help of inequality (\eqref{m8}) it easily follows that:
\beq\label{m9}
\frac{\partial u}{\partial\nu}({\bf y}_0)\leq - 2\lambda A 
r_0^2 e^{-\lambda r_0^2}<0.
\eeq
On the other hand now, ${\bf y}_0$ is an interior point at which 
$u$ reaches a minimum ($u({\bf x})\geq 0 $ on $\Omega$, 
$u({\bf y}_0)=0 $); this entails $\nabla u({\bf y}_0)=0$, hence 
$ \displaystyle\frac{\partial u}{\partial\nu}({\bf y}_0)=0$, which 
contradicts inequality \eqref{m9}.  This last argument ends the proof.    
\end{proof}
The main result of this section is
\begin{lemma}
\label{m10}

$$
B_\alpha (P_1 - \{0\}) \subset 
\stackrel{\circ}{P_1}.
  $$
\end{lemma}
\begin{proof}
Since Lemma \ref{p41} gives $ B_\alpha (X_1) \subset X_1 $,
it suffices to prove that for any $f\in P_1, f\not\equiv 0$,
 there exists $ c > 0 $, such that
\beq\label{m11}
u({\bf x})\geq cM({\bf x}),\quad \forall {\bf x}\in\Omega
\eeq
where in the above $u=B_\alpha(f)$.
\\
The difficulty here is to lower bound $u$ in a neighborhood of the boundary of $\Omega$.
 
Let $W_0:\Omega\to\mathbb{R},\, W_0({\bf x})=e^{-\lambda\|{\bf x}\|^2}
-e^{-\lambda}$,  with \ $ \lambda > 0 $, and 
\\
$W_1: \Omega\to\mathbb{R},\,W_1({\bf x})=\left[ W_0({\bf x})
\right]^\delta×$.
Observe that $W_0({\bf x})=e^{-\lambda} \left[ 
e^{\lambda( 1 - \|{\bf x}\|^2)} × -1 \right]$;
using now the inequalities
$$
e^z - 1 \geq z \quad \text{ and } \quad e^z - 1 \leq z e^\lambda
\quad \text { with } \quad z = \lambda (1 - \|x\|^2) 
  $$
we deduce
\beq
\label{m12-0}
\lambda  e^{-\lambda} (1- \|{\bf x}\|^2)
\leq \left|W_0({\bf x})\right|\leq 
\lambda(1- \|{\bf x}\|^2), \quad \forall {\bf x} \in \Omega
\eeq
which implies
\beq
\nonumber
\lambda^\delta  e^{-\delta \lambda} M({\bf x})
\leq \left|W_1({\bf x})\right|\leq 
\lambda^\delta M({\bf x}), \quad \forall {\bf x} \in \Omega.
\eeq
Then  $W_1\in X_1$ which triggers $W_1\in L_M^2$.
\\
The followings hold true as well: $\nabla W_0=-2\lambda{\bf x}e^{-\lambda\|{\bf x}\|^2 }$ 
and $\nabla W_1=-2\lambda\delta{\bf x}e^{-\lambda\|{\bf x}\|^2 }
W_0^{\delta-1}$.  Inequality \eqref{m12-0} leads to
$\left|\nabla W_1 \right| \leq c 
M^{1-1/\delta}$.  Finally,  Proposition \ref{sfs31}
gives $W_1\in H_M^1$.  

We take on to proving $L_\alpha(u-A W_1) \equiv f-A L_\alpha(W_1)\geq0$ on 
$\Omega-\Omega_\eta$, where $ A > 0 $ and $\eta\in]0,1[$ will be chosen
later.
As $f\geq0$, we need to prove that $L_\alpha(W)\leq0$. 
One has:
$$
L_\alpha(W_1)=-\Delta W+\displaystyle\left(\frac{\nabla M}{×M}+
{\bf k} \right)\cdot \nabla W+ \left(\frac{\Delta M}{×M}-
\frac{|\nabla M|^2}{M^2×}+\nabla\cdot{\bf k}+\alpha  \right)W.
  $$  
Carrying out the calculations by making explicit $\nabla M$, 
$\Delta M$, etc, leads to 
\beq
\label{m12-1}
\begin{split}
L_\alpha (W_1) & =  2 \lambda \delta n e^{- \lambda \|{\bf x}\|^2}
W_0^{\delta - 1} - 4 \lambda^2 \delta \|{\bf x}\|^2
e^{- \lambda \|{\bf x}\|^2} W_0^{\delta - 1} -
4 \lambda^2 \delta (\delta - 1) \|{\bf x}\|^2
e^{- 2 \lambda \|{\bf x}\|^2} W_0^{\delta - 2}
\\
& + 4 \lambda \delta^2 \frac {\|{\bf x}\|^2} 
{1 - \|{\bf x}\|^2} e^{- \lambda \|{\bf x}\|^2} W_0^{\delta - 1}
- 2 \lambda \delta {\bf k} \cdot {\bf x} 
e^{- \lambda \|{\bf x}\|^2} W_0^{\delta - 1} - \frac {2 n \delta}
{1 - \|{\bf x}\|^2} W_0^{\delta}
\\
& + 4 \delta (\delta - 1) \frac {\|{\bf x}\|^2} 
{(1 - \|{\bf x}\|^2)^2} W_0^{\delta}
- 4 \delta^2  \frac  {\|{\bf x}\|^2} {1 - \|{\bf x}\|^2}
 W_0^{\delta}
+ (\nabla \cdot {\bf k} + \alpha) W_0^{\delta}
\end{split}
\eeq
Denote now $ y = 1 - \|{\bf x}\|^2 $.
Expanding about ``$y$ close to $0$'' leads to 
\\
$W_0=\lambda e^{-\lambda} y \left(1+\displaystyle\frac{\lambda y}{2×}
+\frac{(\lambda y)^2} 6 e^{z_1} \right )  $
with $z_1 \in [0, \lambda]$.   
\\
We then have, for any \ $\gamma>0$,
\beq
\label{m13}
 W_0^\gamma=\lambda^\gamma e^{-\lambda\gamma} y^\gamma
 \left(1+\displaystyle\frac{\lambda\gamma
     y}{2×}+ y^2 h(y, \lambda, \gamma) \right)
\eeq  
where, due to the fact that $ y \in [0, 1] $, $h$ is such that
$$
| h(y, \lambda, \gamma)| \leq {\bar h}
   $$
where ${\bar h}$ is a positive constant depending on $\lambda$ and
$\gamma$.
Next, for any $\gamma>0$, 
\beq
\label{m13p}
\displaystyle e^{- \gamma \lambda\|{\bf x}\|^2}=
\displaystyle e^{- \gamma \lambda} e^{\gamma \lambda y} =
e^{-\gamma \lambda}\left(1+ \gamma \lambda y +
\displaystyle\frac{(\gamma \lambda y)^2}{2×} e^{z_2} \right )
\eeq
with $ z_2 \in [0, \lambda \gamma]$.
\\

Expand the right-hand side of \eqref{m12-1} in power series w.r.t. $y$.  
Using \eqref{m13} and \eqref{m13p} and taking into account the equality
$ \|x\|^2 = 1 - y $, one remarks that the coefficients of the leading term $y^{\delta - 2}$
vanish, so after some lengthy (and awkward) algebra one gets: 
\begin{eqnarray}\label{m14}
 L_\alpha(W_1) & = & a_1({\bf x}, \lambda)y^{\delta-1}+a_2({\bf x},y, \lambda)
y^{\delta}\\
& \text {with} & \\
a_1({\bf x}, \lambda) & = & 2\lambda^\delta e^{-\delta\lambda}\left
(-\delta^2\lambda + 4\delta^2 - \delta{\bf x}\cdot{\bf k} \right) 
\end{eqnarray}
and $a_2$ a function satisfying
$$
|a_2({\bf x},y, \lambda)| \leq {\bar a}_2
  $$
where ${\bar a}_2$ is a positive constant depending in $\lambda$.
Next, 
\\
$ -\delta^2\lambda + 4\delta^2 -\delta{\bf x}\cdot{\bf k} 
\leq-\delta^2\lambda + 4\delta^2 +\delta\|{\bf k}\|_{L^\infty} $
and with a suitable choice for $\lambda$, such as:
\beq\label{m16}
\lambda=\frac{1}{\delta^2×}\left(4\delta^2+\delta\|{\bf k}\|_
{L^\infty}+1 \right) 
\eeq 
one gets $4\delta^2-\delta^2\lambda-\delta{\bf x}\cdot{\bf k} \leq -1$, which gives
\beq\label{m16}
a_1({\bf x}, \lambda)\leq -2\lambda^\delta e^{-\delta\lambda}.
\eeq
Therefore 
$$
L_\alpha(W_1) \leq - 2 \lambda^\delta e^{- \delta \lambda}
y^{\delta - 1} + {\bar a}_2 y^\delta = - 2 \lambda^\delta 
e^{- \delta \lambda} y^{\delta - 1} \left ( 1 - \frac {{\bar a}_2} 2
\lambda^{- \delta} e^{\delta \lambda} y \right ).
  $$
Then one may take $y$ small enough (i.e. $\|{\bf x}\|$
close to 1) such that \
$L_\alpha(W_1) \leq 0$.  
\\
It has thus been proved that $\exists \eta\in]0,1[$, close to $1$, 
such that 
\beq
\label{m16-1}
L_\alpha(u-A W_1)\geq0,\quad \forall {\bf x}\in \Omega-
\Omega_\eta,\quad \forall A>0.
\eeq 
Next, from Lemma \ref{m1} 
we have $u>0$ over $\Omega$. Since $u$ is also  continuous,
$A$ may be chosen such that $\displaystyle \mathop 
{\min}_{{\bf x}\in \overline{\Omega}_\eta}u({\bf x})\geq A 
\underbrace{\displaystyle \mathop {\max}_{{\bf x}\in 
\overline{\Omega}_\eta}W({\bf x})}_{\leq \,2^\delta}$.  
Take  $A=\displaystyle\dfrac{1}{2^\delta}\mathop {\min}_{{\bf x}
\in \overline{\Omega}_\eta}u({\bf x})$.  Such a choice leads to 

\beq\label{m17}
u({\bf x})\geq A W_1({\bf x}), \,\forall {\bf x}\in \overline{\Omega}_\eta
\eeq

Since $u,\, W_1 \in H_M^1$, use of equations \eqref{m16-1}, \eqref{m17} and of Comparison
Principle allows one to infer that $u({\bf x})\geq 
AW_1({\bf x}), \,\forall {\bf x}\in \Omega-\Omega_\eta$.  One more use 
of \eqref{m17} 
implies that, in fact, this inequality holds true  on the entire $\Omega$ domain. 
Now the inequality \eqref{m12-0} gives the result.

\end{proof}

\section{Proof of the main result}
\label{proof}

We are now in position to give the proof of Theorem \ref{est9-1}.
\\
{\bf Step 1}.
\\
From part $(ii)$ in Lemma \ref{p1}  one infers
$$
B_\alpha \in \mathscr{L}(X_1, \; X_{1/2 - 1/\delta}).
  $$
From Lemma \ref{p25} we obtain
$$
B_\alpha \in \mathscr{L}(X_{1/2 + (2j-1)/\delta}, \; 
X_{1/2 + (2j+1)/\delta}), \quad \text { for any } \;
j \in \{0, 1, \cdots j_0 \}
  $$
where
$$
j_0 = \max \left \{ j \in \mathbb{N}, \quad \frac 1 2 + \frac {2j+1}
\delta < 1 \right \}
  $$
(in other words, $j_0$ is the unique natural number belonging
to the interval \ $[\frac \delta 4 - \frac3 2 , \frac \delta 4 - \frac 1 2 [$).
\\
Due to the inequality
$ \frac 1 2 + \frac {2 j_0 + 1} \delta \geq 1 - \frac 2 \delta $ we have the inclusion
$ X_{1/2 + (2 j_0 + 1)/\delta} \subset X_{1 - 3/\delta} $. Using again Lemma \ref{p25} we
obtain
$$
B_\alpha^{j_0 + 2} \in \mathscr{L}(X_1, \; X_{1 - 1/\delta}).
  $$
Finally, from Lemma \ref{p41} we deduce
$$
S_\alpha \in \mathscr{L}(X_1, \; X_1)
  $$
where the following notation has been used:
$$
S_\alpha = B_\alpha^{j_0 + 3}.
   $$

As $B_\alpha$ is a compact operator that maps $X_1$ onto $X_{1/2 - 1/\delta}$ (see part $(ii)$ in Lemma \ref{p1}) it follows that $S_\alpha$ is compact as well.
\\
On the other hand, Lemma \ref{m10} gives
\beq
\label{incl-S}
S_\alpha(P_1 - \{0\}) \subset \stackrel{\circ}{P×}_1.
\eeq

We are now in a position that allows to make use of the strong version of the Krein-Rutman theorem
(see for example \cite{d-lio})
to the operator $S_\alpha$, the Banach space $X_1$ and the cone
$P_1$. One deduces the existence of an eigenvalue $ \mu_0 > 0 $
of $S_\alpha$, to which corresponds the eigenvector 
$u_0 \in \stackrel{\circ}{P×}_1$, that is 
\beq
\label{eigen-S}
B_\alpha^{j_0 + 3} u_0 = \mu_0 u_0.
\eeq
Moreover, if $u_1 \in \stackrel{\circ}{P×}_1 $ is any other
eigenvector of $S_\alpha$, related to a positive
eigenvalue, then $u_1$ is equal to $u_0$ up to a
multiplicative positive constant. 

{\bf Step 2} (Existence).
\\
Denote by $\text{Spr}(B_\alpha)$ the spectral radius of
$B_\alpha$, where $B_\alpha$ is considered an element of
$\mathscr{L}(L_M^2)$.
It is well known (see for example Section III.6.2 in \cite{kato}) that 
$$
\text{Spr}(B_\alpha) = \lim_{m \rightarrow + \infty}
\| B_\alpha^m \|_{\mathscr{L}(L_M^2)}^{1/m}.
  $$
Since $u_0 \in  L_M^2 $ and $ u_0 \neq 0 $, using 
(\eqref{eigen-S}) one has:
$$
\| B_\alpha^{m (j_0 + 3)} \|_{\mathscr{L}(L_M^2)} \geq
\frac { \| B_\alpha^{m (j_0 + 3)} u_0 \|_{L_M^2} }
{ \| u_0 \|_{L_M^2} } = \mu_0^m
  $$
which triggers
\beq
\label{sp-rad}
\text{Spr}(B_\alpha) \geq \mu_0^{1/(j_0 + 3)} > 0.
\eeq
On the other hand, let us denote by $P_M^2$ the 
(reproducible) cone of
positive functions in $L_M^2$.
Due to the Weak Maximum Principle (Lemma \ref{est3}), 
$$
B_\alpha (P_M^2) \subset P_M^2.
  $$
We also have that $B_\alpha$ is a compact operator from
$L_M^2$ onto itself (due to the compact embedding of 
$H_M^1$ in $L_M^2$, see \cite{ne}).
\\
Use now the weak version of the Krein-Rutman theorem
(see \cite{d-lio}) for the operator $B_\alpha$, 
the Banach space $L_M^2$ and the cone $P_M^2$.
It turns out there exists an eigenvalue $\tilde{\mu}_0 > 0$ of 
$ B_\alpha$ and a corresponding eigenvector
$ \tilde{u}_0 \in P_M^2 - \{0\} $, i.e.
$$
 B_\alpha  \tilde{u}_0 = \tilde{\mu}_0 \tilde{u}_0.
  $$
Clearly $\tilde{u}_0$ also belongs to $H_M^1$.
Moreover, any other eigenvalue $\mu \in \mathbb{C}$ of
$B_\alpha$ is such that
\beq
\label{maj-vp}
|\mu| \leq \tilde{\mu}_0.
\eeq
Let us now denote 
$ \tilde{\lambda}_0 = 1/\tilde{\mu}_0 - \alpha $;  $ \tilde{\lambda}_0$ is 
clearly an eigenvalue of $L$ related to the same eigenfunction
$\tilde{u}_0$.  It then follows:\\

$$
\int_\Omega M \nabla \left ( \frac {\tilde{u}_0} M \right )
\cdot \nabla \left ( \frac \varphi M \right ) \ud {\bf x}
- \int_\Omega {\bf k} \tilde{u}_0
\cdot \nabla \left ( \frac \varphi M \right ) \ud {\bf x}
= \tilde{\lambda}_0
\int_\Omega \frac { \tilde{u}_0 \varphi } M \ud {\bf x}, \quad \forall \;
\varphi \in H_M^1.
  $$

Set $ \varphi = M $ in the above equation, and since
$ \displaystyle\int_\Omega \tilde{u}_0 > 0 $ ( remark that
$ \tilde{u}_0 \in L^1(\Omega) $ by the obvious embedding
$L_M^2(\Omega)\displaystyle\mathop{\subset}_{\text{cont}}
L^1(\Omega) $),
we deduce that $\tilde{\lambda}_0 = 0$.
Then the following expression:
$$
\psi = \frac b {\displaystyle\int_\Omega \tilde{u}_0 \ud {\bf x}} \tilde{u}_0 
  $$
gives a solution of eq\eqref{vp1}.
\\
{\bf Step 3} (Uniqueness).
\\
Assume $\psi_1$ and $\psi_2$ are
two solutions to the problem \eqref{vp1}-\eqref{vp3}.
Then $\psi_1$ and $\psi_2$ are non-negative eigenvectors of 
operator $L$ corresponding to the eigenvalue 0. This implies
$$
B_\alpha \psi_m = \frac 1 \alpha \psi_m, \quad m= 1, 2
  $$
which gives
$$
S_\alpha \psi_m = \left ( \frac 1 \alpha \right )^{j_0 + 3} 
\psi_m, \quad m= 1, 2.
  $$
We then obtain $\psi_m \in X_1$ and
by (\eqref{incl-S}) we also have $ \psi_m \in \stackrel{\circ}{P×}_1, 
m = 1,2.$
Now by the uniqueness property of the strong version of the
Krein-Rutman theorem, there exists $ r > 0$ such that
$ \psi_1({\bf x}) = r \psi_2({\bf x}), \; \; \forall \,{\bf x}
\in \Omega$. Now since $\displaystyle\int_\Omega \psi_1 =\displaystyle\int_\Omega \psi_2
= b $  we obtain $r = 1$ which proves the uniqueness.

\begin{remark}
From inequality (\eqref{maj-vp}) one can deduce, proceeding in a classical manner, 
that 
$$
\text{Re}(\lambda) \geq 0 
  $$ 
for any other complex eigenvalue of operator $L$.
The eigenvalue 0 is then the {\bf principal eigenvalue} of the
operator $L$. Moreover, what is quite remarkable is the fact that $0$ is the principal eigenvalue of $L$ {\bf for any} function $ {\bf k}
\in (W^{1, \infty}(\Omega))^n $.   

\end{remark}

\section{Final comments}

We have offered a proof to the fact that the FENE dumbbell configurational distribution function diffusion equation - see the corresponding boundary value problem described in eqs \eqref{bp31}-\eqref{bp34}- has unique steady state solutions.  In doing so, we relied on the Krein - Rutman theory of elliptic operators.  

There are several motivations for this work.  In \cite{bird2} asymptotic solutions for the probability density diffusion equation - valid for slow flows - are presented, but no proof for the existence of such solutions is offered.  While in this work we have proved the existence of solutions to the diffusion equation for slow and fast flows (that is irrespective of whether the velocity gradient is ``small'' or ``large''), the questions related to the convergence of explicit expansion solutions given in \cite{bird2} (and in what functional space it occurs) are still to be addressed to the fullest.  Moreover, we expect our results to further the work in finding asymptotic solutions valid for ``large'' velocity gradients, i.e. for fast flows, for the FENE dumbbell model; a solution in such a case is known only for rigid dumbbells (see \cite{otf}).  

Now, the elastic (or rigid) dumbbell polymer chain models are certainly crude representations of the real chains.  That set aside, they do capture several - but not all - important features of viscoelastic flows (e.g. shear rate dependent viscosity, first normal stress difference).  Moreover, they owe a certain popularity among polymer scientists (from experimentalists to applied mathematicians) to their relative simplicity.  More realistic models use different chain representations.  For instance, Doi and Edwards \cite{de} developed the so-called tube model for melt systems, that makes use of the de Genne's reptation ideas of anisotropic chain diffusion.  Schweitzer and co-workers \cite{ken} developed a mode-coupling model in which the intermolecular structural constraints upon the motion of single macromolecules are modeled as a many body caging effect.  Ngai and Plazek developed their own coupling model \cite{np,ng}, very successful in predicting the thermo-rheological complexity.  This being said, bead-spring or bead-rod chain models still attract significant attention: see for example \cite{sch,siz,gh,cif,und,kro}.  For sure their full capabilities are still to be uncovered.

\section{Acknowledgements}

The authors thank Dr Jean-Francois Palierne, Laboratoire de Physique, Ecole Normale Sup\'erieure de Lyon, and Dr Vitaly Volpert, Directeur de Recherche au CNRS, Universit\'e Lyon-I, for useful talks on polymer molecular dynamics, and on elliptic operators, respectively.

\end{document}